\documentclass[a4paper,12pt]{extarticle}
\usepackage{latexsym}
\usepackage{amscd}
\usepackage{graphics}
\usepackage{amsmath}
\usepackage{amssymb}
\usepackage{bbold}
\usepackage{mathrsfs}
\usepackage{amsthm}
\usepackage{xcolor}
\usepackage{accents}
\usepackage{enumerate}
\usepackage{url}
\usepackage{tikz-cd}
\usetikzlibrary{decorations.pathreplacing}
\usepackage{marginnote}
\usepackage{hyperref}
\usepackage{multicol,tikz}
\usetikzlibrary{calc}
\usepackage{marvosym}

\usepackage{newpxtext} \usepackage[euler-digits]{eulervm}

\theoremstyle{plain}
\newtheorem{theorem}{\sc Theorem}[section]
\makeatletter
\newcommand{\settheoremtag}[1]{% \settheoremtag{<tag>}
  \let\oldthetheorem\thetheorem% Store \thetheorem
  \renewcommand{\thetheorem}{#1}% Redefine it to a fixed value
  \g@addto@macro\endtheorem{% At \end{theorem}, ...
    \addtocounter{theorem}{0}% ...restore theorem counter value and...
    \global\let\thetheorem\oldthetheorem}% ...restore \thetheorem
  }
\newtheorem{prop}[theorem]{\sc Proposition}
\newtheorem{lem}[theorem]{\sc Lemma}

\theoremstyle{definition}
\newtheorem{defn}[theorem]{\sc Definition}

\newtheorem{qu}[theorem]{\sc Question}
\newtheorem{ex}[theorem]{\sc Example}

\DeclareMathOperator{\R}{\mathbb{R}}
\newcommand{\abs}[1]{\left\lvert#1\right\rvert}
\newcommand{\norm}[1]{\left\lVert#1\right\rVert}
\newcommand{\op}[1]{\operatorname{#1}}

\numberwithin{equation}{section}

\allowdisplaybreaks

\title{A survey on orderability and contact non-squeezing}
\author{Igor Uljarevi\'c}
\date{September, 2024}

\usepackage{biblatex}
\begin{document}
% created the project on May 8, 2024

\maketitle

\begin{abstract}
The present article provides an overview of Yakov Eliashberg's seminal contributions to the concepts of orderability and contact non-squeezing. It also examines subsequent research by various authors, highlighting the significance of these notions and offering a detailed account of the current state of the field.
\end{abstract}

This survey article about orderability and contact non-squeezing is a part of the Celebratio volume honouring the work of Yakov Eliashberg. The concept of orderability in contact geometry was introduced by Eliashberg and Polterovich \cite{eliashberg2000partially} in an attempt to find an analogue of the Hofer metric in contact geometry. A couple of years after the notion of orderability was introduced, Eliashberg, Kim, and Polterovich discovered a connection between orderability and a rigidity phenomenon in contact geometry that is now known as contact non-squeezing \cite{eliashberg2006geometry}. The first instance of contact non-squeezing, however, dates back to a 1991 article by Eliashberg in which the solid tori in the standard $\mathbb{S}^3$ were classified up to a contactomorphism\cite{eliashberg1991new}. 

The seminal papers \cite{eliashberg2000partially} and \cite{eliashberg2006geometry} led to intense research over the past 20 years that the present article gives an overview of.  For the sake of clarity and brevity, we focus on geometric results,  only briefly mentioning the techniques and methods used. The exposition of this article does not follow chronological development of the area. We start with the concept of contact non-squeezing, which is intuitive even without formal mathematical training, and motivate the orderability by its connection with contact non-squeezing. We then proceed to develop topics related to orderability, such as its relation with quasimorphisms, overtwistedness, and orderability of isotopy classes of Legendrians.

\section{Motivation behind contact non-squeezing}

When approaching an uncharted territory, an unknown realm, or a new mysterious object, it is natural (and, perhaps, only reasonable), to attempt to relate it to something familiar, something well understood. In doing so, it is also quite reasonable to seek the features in which the new unknown differs from the old known and to recognise the properties that the unknown and the known have in common. This general abstract principle, applied to contact geometry (unknown) and smooth topology (known\footnote{Admittedly, considering many mysteries still present in smooth topology, calling it  \emph{known} seems inappropriate. We do so here, nevertheless, understanding `known' in relative terms. }), sparks a growing interest in flexibility and rigidity phenomena in contact geometry. Are contact manifolds almost as flexible as smooth manifolds or do they exhibit rigidity much like Riemannian manifolds? The notion of size is a good testing ground for rigidity versus flexibility, for the size does not exist in the eyes of smooth topology. In contact geometry, the notion of size can be conveniently addressed by \emph{contact non-squeezing}. The precise definition of contact non-squeezing will be given in the next section. Intuitively, contact non-squeezing allows one to compare subsets of contact manifolds. In other words, using contact non-squeezing one is able to make sense of what it means that one subset is bigger than the other in contact geometry.

Contact non-squeezing is to a great extent also motivated by the success of Gromov's non-squeezing in symplectic geometry \cite{gromov1985pseudo}. It is clear, however,  that symplectic geometry sees the size because, as opposed to contact manifolds, every symplectic manifold has a natural volume form. 
The role of symplectic non-squeezing is in showing that  size in symplectic geometry is not completely determined by the volume and that symplectic geometry cannot be reduced to the volume-preserving geometry.

\section{Definition of contact non-squeezing and first examples}
The notion of contact non-squeezing was introduced in the seminal paper \cite{eliashberg2006geometry} by Eliashberg, Kim, and Polterovich. 

\begin{defn}[Eliashberg-Kim-Polterovich]\label{def:squeezing}
Let $\Omega_1$ and $\Omega_2$ be open subsets of a contact manifold $M$. The subset $\Omega_1$ can be (contactly) squeezed into the subset $\Omega_2$ if there exists a contact isotopy $\phi_t: \overline{\Omega}_1\to M, t\in[0,1]$ such that $\phi_0$ is equal to the identity and such that $\phi_1(\overline{\Omega}_1)\subset\Omega_2.$
\end{defn}

\begin{figure}
	\begin{center}
		\includegraphics[scale=0.7]{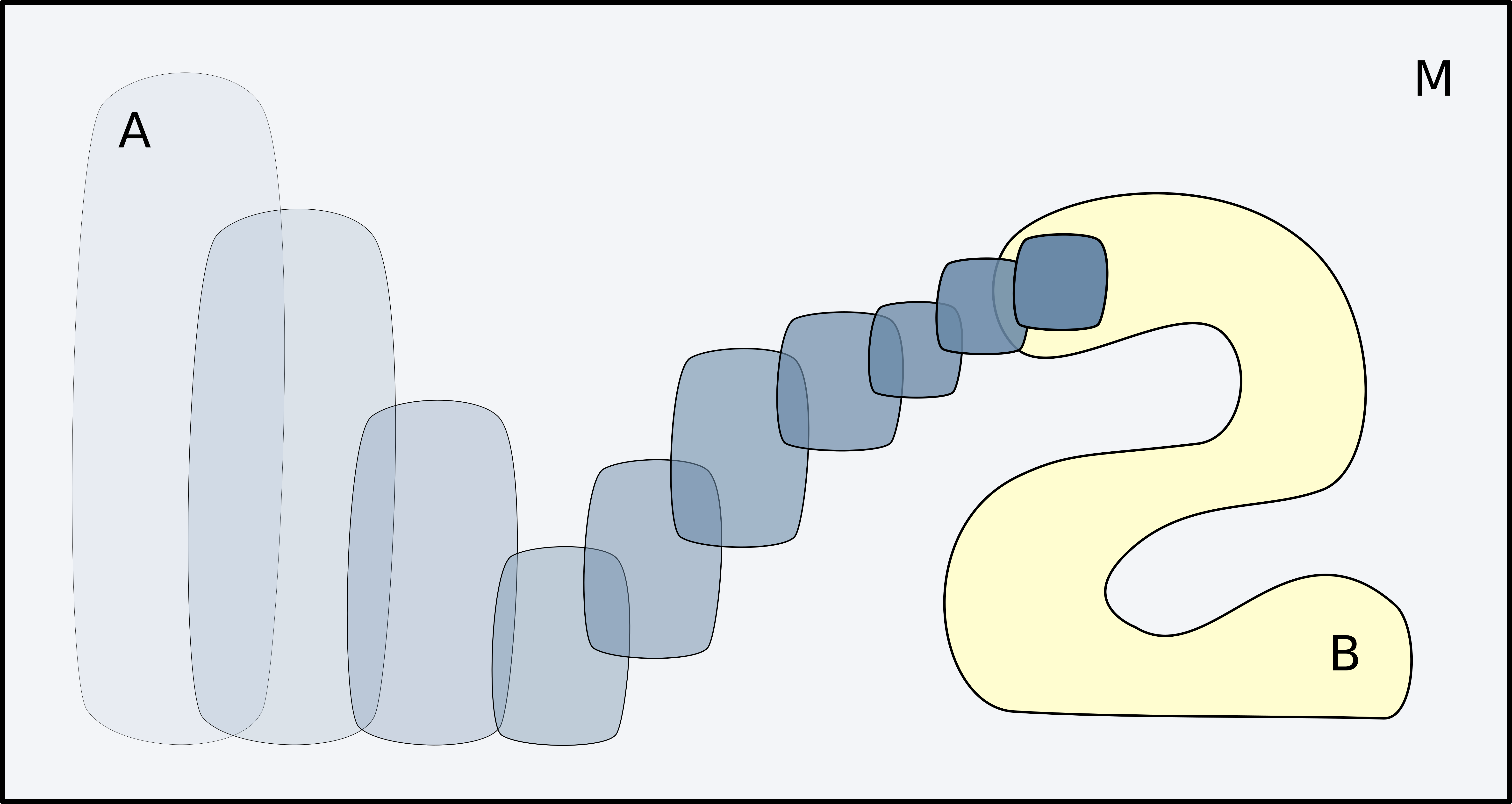}
	\end{center}
	\caption{Squeezing of $A$ into $B$}
\end{figure}
	
In Definition~\ref{def:squeezing}, $\overline{\Omega}$ denotes the closure of the subset $\Omega$. If $\overline{\Omega}_1$ is a compact set, then the contact isotopy $\phi_t$ from Definition~\ref{def:squeezing} extends to a compactly supported isotopy of $M$. 
The following example shows that the contact non-squeezing in the standard contact $\R^{{2n+1}}$  does not differ from the non-squeezing in the smooth category.   
\begin{ex}\label{ex:standardR}
Consider $\R^{2n+1}$ with the standard contact structure \[\xi:=\ker\left(  dz + \sum_{j=1}^n x_jdy_j\right).\]
The map $\phi_s: \R^{2n+1}\to\R^{2n+1}$ defined by $\phi_s(x,y,z)= (sx, sy, s^2z)$ is a contactomorphism for all $s\not=0.$ 
\end{ex}
In fact, since contact manifolds are all modelled on the standard contact $\R^{2n+1},$ the example above actually shows that locally contact geometry is as oblivious to size as smooth topology is. More precisely, any subset of a contact manifold whose closure is contained in a contact Darboux chart can be contactly squeezed into an arbitrary non-empty open subset. An extreme situation, along these lines, is that of the standard contact sphere where the contact Darboux chart covers everything but a single point. 
Thus, the contact squeezing on the standard contact sphere is also trivial. The precise formulation is given in the following example.

\begin{ex}
In the standard contact sphere $\mathbb{S}^{2n+1}$, any non-dense subset can be contactly squeezed into any non-empty open subset.
\end{ex}

The most basic examples of contact manifolds do not admit non-trivial contact non-squeezing. In addition, contact non-squeezing is trivial on a small scale for every contact manifold. It is, therefore, a surprise that there exist contact manifolds that exhibit non-trivial contact non-squeezing. 

\section{Non-squeezing on prequantization bundles}

The first example of a contact non-squeezing that is not also topological non-squeezing was discovered by Eliashberg \cite{eliashberg1991new} fifteen years before the formal introduction of the notion of contact non-squeezing. In \cite{eliashberg1991new}, Eliashberg classified the solid tori in the standard contact $\mathbb{S}^3$ up to a contactomorphism using the shape invariant. The following theorem is a direct consequence of this result.

\begin{theorem}[Eliashberg]\label{thm:n=1}
If $R>r>0$ then the set $B(R)\times\mathbb{S}^1$ cannot be squeezed inside $B(r)\times\mathbb{S}^1$ by a compactly supported contact isotopy of $\R^2\times\mathbb{S}^1$.
\end{theorem}
In the theorem, $B(r)$ denotes the ball $\{\pi\abs{z}^2<r\}$ of radius $\sqrt{\frac{r}{\pi}}$ in $\R^2$, and $\R^2\times\mathbb{S}^1$ is the contact manifold obtained by quotienting the standard contact $\R^3$ by the $\mathbb{Z}$-action $k\bullet (x,y,z):=(x,y,z+k)$. Obviously, the theorem describes a purely contact-geometric phenomenon because $B(R)\times\mathbb{S}^1$ can be smoothly squeezed into $B(r)\times\mathbb{S}^1$ for any choice of positive numbers $R$ and $r$. The ideas from \cite{eliashberg1991new} are further developed in \cite{cant2024contact}. 

Theorem~\ref{thm:n=1} is to some extent misleading. Namely, it is only in dimensions greater than three that quantum-like behaviour of contact non-squeezing emerges. In higher dimensions, there exists a threshold that separates flexibility from rigidity. If $R$ is less than the threshold, then $B(R)\times\mathbb{S}^1$ can be contactly squeezed as much as we please. On the other hand, if $R$ is not smaller than the threshold, then $B(R)\times\mathbb{S}^1$ cannot be squeezed even in itself. 

\begin{theorem}[Eliashberg-Kim-Polterovich, Chiu]\label{thm:ngeq2}
Let $n>1$ be an integer and let $R\geqslant r>0$. Then the set $B(R)\times\mathbb{S}^1\subset \R^{2n}\times\mathbb{S}^1$ can be squeezed inside $B(r)\times\mathbb{S}^1$ if, and only if, $R<1.$
\end{theorem}
The number 1 has a prominent role in Theorem~\ref{thm:ngeq2}. It is the threshold separating flexibility and rigidity. The reason why 1 is the threshold and not some other number is related to the size of the fibre $\mathbb{S}^1$. In our conventions, the contact structure on $\mathbb{R}^{2n}\times\mathbb{S}^1$ is furnished by a contact form whose Reeb flow has period 1. A substantial part of Theorem~\ref{thm:ngeq2} was proved by Eliashberg, Kim, and Polterovich in \cite{eliashberg2006geometry}. They showed that $B(R)\times\mathbb{S}^1$ can be squeezed into $B(r)\times\mathbb{S}^1$ if $R<1$. They also showed the following non-squeezing result that, in particular, implies that $B(R)\times\mathbb{S}^1$ cannot be squeezed into $B(r)\times\mathbb{S}^1$ if  the interval $[r, R]$ contains an integer. In the next theorem, $C(r)$ denotes the symplectic cylinder
\[C(r):= B(r)\times\R^{2n-2}\subset \R^2\times\R^{2n-2}.\]

\begin{theorem}[Eliashberg-Kim-Polterovich]\label{thm:cylinder}
Let $R\geqslant r>0$ be such that the interval $[r, R]$ contains an integer. Then, there does not exist a compactly supported contactomorphism $\phi:\R^{2n}\times\mathbb{S}^1\to\R^{2n}\times\mathbb{S}^1$ that maps the closure of $B(R)\times\mathbb{S}^1$ into $C(r)\times\mathbb{S}^1$.
\end{theorem}

The rest of Theorem~\ref{thm:ngeq2} was completed by Chiu \cite{chiu2017nonsqueezing} using microlocal analysis. Theorem~\ref{thm:ngeq2} established itself as an important milestone in the study of contact geometry. There are several works that test whether certain techniques used in contact geometry are powerful enough to detect phenomena described by Theorem~\ref{thm:ngeq2}. For instance, Sandon \cite{sandon2011contact} used generating functions to prove the non-squeezing of Theorem~\ref{thm:ngeq2} for the radii $R$ and $r$ that are separated by an integer. The same is proved by Albers and Merry \cite{albers2018orderability} and by Cant and the author \cite{cant2024selective} using respectively the Rabinowitz Floer homology and a local version of symplectic homology. In \cite{fraser2016contact}, Fraser proved a part of Theorem~\ref{thm:ngeq2} not covered in \cite{eliashberg2006geometry} using SFT techniques, thus providing an alternative, more in line with the methods of  \cite{eliashberg2006geometry}, to Chiu's sheaf-theoretic proof.
Yet another alternative to Chiu's proof is presented in \cite{fraser2023contact}, as a continuation to \cite{sandon2011contact},  by Fraser, Sandon, and Zhang.  

Theorem~\ref{thm:cylinder} extends Theorem~\ref{thm:ngeq2} in two ways:
\begin{enumerate}
	\item it obstructs squeezing of $B(R)\times\mathbb{S}^1$ into $C(r)\times\mathbb{S}^1$, which is bigger than $B(r)\times\mathbb{S}^1$,
	\item it proves non-squeezing by a compactly supported contactomorphism of $\R^{2n}\times\mathbb{S}^1$, as opposed to a compactly supported contact isotopy. 
\end{enumerate}
A squeezing by a compactly supported contactomorphism is occasionally called \emph{coarse} \cite{fraser2016contact,fraser2023contact}. It is not clear whether there is a difference between coarse and non-coarse squeezing in $\R^{2n}\times\mathbb{S}^1$ at all. This issue is related to the following important open question.
\begin{qu}
Does there exist a compactly supported contactomorphism $\R^{2n}\times\mathbb{S}^1\to \R^{2n}\times\mathbb{S}^1$ that is not isotopic to the identity through compactly supported contactomorphisms?
\end{qu}

The following example shows that the conditions about contact supports of contactomorphisms and contact isotopies in Theorems~\ref{thm:ngeq2} and \ref{thm:cylinder} are essential.
\begin{ex}[Proposition~1.24 in \cite{eliashberg2006geometry}]
Let $N\in\mathbb{N}$ be a positive integer. After identifying $\R^{2n}$ with $\mathbb{C}^n$ and $\mathbb{S}^1$ with $\R/\mathbb{Z}$, denote by $F_N:\mathbb{R}^{2n}\times\mathbb{S}^1\to \mathbb{R}^{2n}\times\mathbb{S}^1$ the map
\[F_N(z,\theta):= \left(  \frac{e^{{2\pi Ni\theta}}}{\sqrt{1 + N\pi\abs{z}^2}}\cdot z , \theta \right).\]
Then, the map $F_N$ is a contactomorphism. For any given $R, r>0$ there exists $N\in\mathbb{N}$ such that the map $F_N$ squeezes $B(R)\times\mathbb{S}^1$ into $B(r)\times\mathbb{S}^1$.
\end{ex}

It is a natural question whether Theorem~\ref{thm:ngeq2} can be extended to prequantizations over symplectic manifolds other than the standard $\R^{2n}$.  This question is addressed by Eliashberg, Kim, and Polterovich in \cite[Proposition~1.26]{eliashberg2006geometry} and by Albers and Merry \cite[Theorem~1.22]{albers2018orderability} for the prequantizaton $M\times\mathbb{S}^1$ of a Liouville manifold $M$. Both of the results \cite[Proposition~1.26]{eliashberg2006geometry} and \cite[Theorem~1.22]{albers2018orderability} are somewhat abstract in that they involve Floer-theoretic conditions, respectively in terms of contact homology and in terms of  symplectic capacities, which are not always easily checked. There are no published results regarding prequantizations of closed symplectic manifolds as of the time of writing\footnote{After this survey article was submitted, Pierre-Alexandre Arlove posted a preprint on ArXiv \cite{arlove2024contact}, which proves contact non-squeezing in lens spaces.}. The author has been informed about two works in progress by Albers, Shelukhin, and Zapolsky \cite{zapolsky} and by Rizell and Sullivan \cite{rizell} that resolve the question of contact squeezing on prequantizations over rich classes of closed symplectic manifolds.
Along these lines, \cite{uljarevic2023selective} proves contact non-squeezing on homotopy spheres admitting a filling with large symplectic homology. These homotopy spheres include many Brieskorn manifolds, and among them the Ustilovsky spheres, which can be thought of as prequantizations of certain symplectic orbifolds. The results of \cite{uljarevic2023selective} have, to some extent, different flavour  in that they establish non-squeezing of a smoothly embedded ball as opposed to a solid torus $B(R)\times\mathbb{S}^1$.

Work of Serraille and Stojisavljevi\'c \cite{vukasin} shows strong indications that there exists a relation between contact non-squeezing and the Rokhlin property of the group of compactly supported contactomorphisms.

\section{The role of positive loops in contact squeezing}\label{sec:roleloops}

In this section, we discuss some details about the flexibility part of Theorem~\ref{thm:ngeq2}. Our primary goal is to emphasise the role played by positive loops of contactomorphisms in contact squeezing. A path of contactomorphisms of a cooriented contact manifold is called \emph{positive} if it is generated by a positive contact Hamiltonian. In particular, a loop of contactomorphisms is called \emph{positive} if it is positive as a path. Given a positive path $\varphi_t:\mathbb{S}^{2n-1}\to\mathbb{S}^{2n-1}$ of contactomorphisms furnished by a 1-periodic contact Hamiltonian $h_t:\mathbb{S}^{2n-1}\to\R^+$, one can associate to it a (fibrewise) star-shaped domain
\[U(\varphi):= \left\{(z,t)\:|\: H_t(z)<1\right\}\subset \mathbb{C}^n\times\mathbb{S}^1.\]
Here, $H_t:\mathbb{C}^n\to\R$ is a function given by 
\[H_t(z):= \left\{\begin{matrix} \norm{z}^2\cdot h_t\left( \frac{z}{\norm{z}}\right) & \text{for }z\not=0\phantom{.}\\ 0& \text{for }z=0.\end{matrix}\right.\]
The function $H_t$ can be thought of as the continuous extension to $\mathbb{C}^n$ of the $\R^+$-equivariant Hamiltonian on the symplectization $S\mathbb{S}^{2n-1}\approx\mathbb{C}^n\setminus\{0\}$ that corresponds to $h_t$. An important property of the correspondence $\varphi\leftrightarrow U(\varphi)$ is expressed by the following lemma \cite[Lemma~1.21]{eliashberg2006geometry}. It claims that homotopic positive paths, relative endpoints, furnish contact isotopic star-shaped domains.

\begin{lem}
Let $\varphi^s, s\in[0,1]$ be a homotopy of positive paths of contactomorphisms $\mathbb{S}^{2n-1}\to \mathbb{S}^{2n-1}$ with fixed endpoints. That is, for each $s$, $\varphi^s_t:\mathbb{S}^{2n-1}\to\mathbb{S}^{2n-1}$ is a positive path of contactomorphisms and $\varphi^s_0= \varphi_0^0, \varphi_1^s=\varphi_1^0$. Then, there exists a contact isotopy $\psi_s:\mathbb{C}^n\times\mathbb{S}^1\to\mathbb{C}^n\times\mathbb{S}^1$ such that $\psi_s(U(\varphi^0))= U(\varphi^s)$. 
\end{lem}

Another ingredient required for the proof that $B(R)\times\mathbb{S}^1$ can be squeezed into itself when $R<1$ is the existence of a contractible positive loop of contactomorphisms $\mathbb{S}^{2n-1}\to\mathbb{S}^{2n-1}$ for $n\geqslant 2$.

\begin{lem}\label{lem:EposloopS}
Let $n\geqslant 2$. Then, there exists a contractible positive loop of contactomorphisms $\varphi_t:\mathbb{S}^{2n-1}\to\mathbb{S}^{2n-1}$.
\end{lem}

The lemma follows from the work of Olshanskii \cite{olshanski1982invariant}. An alternative proof can be found in \cite{eliashberg2006geometry}. Now, we sketch the proof that $B(R)\times\mathbb{S}^1\subset \mathbb{C}^n\times\mathbb{S}^1$ can be squeezed into itself for $n>2$ and for $R>0$ sufficiently small. This claim is a part of Theorem~\ref{thm:ngeq2}.

\begin{prop}
Let $n>2$ be an integer. Then, there exists a sufficiently small $R>0$ such that $B(R)\times\mathbb{S}^1\subset \mathbb{C}^n\times\mathbb{S}^1$ can be squeezed into itself.
\end{prop}
\begin{proof}
The proof follows Section~1.7 from \cite{eliashberg2006geometry}. By Lemma~\ref{lem:EposloopS}, there exists a contractible positive loop $\varphi_t:\mathbb{S}^{2n-1}\to\mathbb{S}^{2n-1}$ of contactomorphisms. Let $\varphi_t^s, s\in[0,1]$ be a homotopy from the identity to $\varphi_t$. In particular, $\varphi_t^0$ is equal to the identity and $\varphi_t^1=\varphi_t$ for all $t$. Denote by $f^s$ the contact Hamiltonian of the contact isotopy $\varphi_t^s$ and by $F^s_t:\mathbb{C}^n\to\R$ the function given by
\[ F_t^s(z):= \left\{\begin{matrix} \norm{z}^2\cdot f^s_t\left( \frac{z}{\norm{z}}\right) & \text{for }z\not=0\phantom{.}\\ 0& \text{for }z=0.\end{matrix}\right. \]
There exists $\mu>0$ such that $F_t^s(z)>-\mu\pi\norm{z}^2$ for $z\not=0$. Define $E(z):= \pi\norm{z}^2.$ Denote by $\phi^H$ the Hamiltonian isotopy of the Hamiltonian $H$ and by $H\#G$ the Hamiltonian of the composition $\phi^H_t\circ \phi_t^G$.  The inequality
\begin{align*}
\left( \frac{E}{R}\# F^s\right)_t (z)= & \frac{1}{R}\cdot E(z) + F_t^s\circ (\phi_t^E)^{-1}(z)\\
= & \frac{\pi}{R}\cdot \norm{z}^2 + F_t^s\left(e^{\frac{2\pi t}{R}} z\right)\\
> & \left( \frac{1}{R} - \mu \right)\pi\norm{z}^2
\end{align*}
implies that $(\frac{1}{R}E)\# F^s$ generates a positive path of contactomorphisms for each $s$ if $\frac{1}{R}>\mu$. Therefore, the sets
\[U\left( \left\{ e^{-\frac{2\pi t}{R}}\cdot \varphi^0_t \right\} \right)= B\left(R\right)\times\mathbb{S}^1\]
and
\[U\left( \left\{ e^{-\frac{2\pi t}{R}}\cdot \varphi^1_t \right\} \right) \subset \operatorname{int} B\left(R\right)\times\mathbb{S}^1\]
are contact isotopic. This proves the proposition up to the following detail: the Hamiltonian of the contact isotopy $e^{\frac{2\pi i t}{R}}\cdot \varphi_t$ might not be 1-periodic and, therefore, the set $U\left( \left\{ e^{-\frac{2\pi t}{R}}\cdot \varphi^1_t \right\} \right) $ might not be well defined, strictly speaking. The original proof \cite[Section~2.1]{eliashberg2006geometry} addresses this issue.
\end{proof}

\section{Orderability in contact geometry}\label{sec:ord}

The considerations of the previous section can be generalized: a contractible positive loop of contactomorphisms on the boundary $\partial W$ of a Liouville domain $W$ furnishes a squeezing of a subset $\Omega\times\mathbb{S}^1\subset W\times\mathbb{S}^1$ inside itself. The existence of a contractible positive loop, or even just a positive loop, is a serious condition that is often not met. The importance of positive loops in contact geometry was first recognised by Eliashberg and Polterovich in \cite{eliashberg2000partially}, where they introduced the notion of orderability (of a contact manifold). 

\begin{defn}\label{def:orderability}
A cooriented contact manifold $M$ is called \emph{orderable} if there are no contractible positive loops $\varphi_t:M\to M$ of contactomorphisms.
\end{defn}

Denote by $\op{Cont}(M)$ the group of contactomorphisms of $M$, and by $\widetilde{\op{Cont}}(M)$ the universal cover of $\op{Cont}(M)$. If a contact manifold $M$ is orderable, then one can define a partial order $\prec$ on $\widetilde{\op{Cont}}(M)$ in the following way: for $A, B\in\widetilde{\op{Cont}}(M)$, the relation $A\prec B$ holds if, and only if, there exists a non-negative path $\varphi_t:M\to M$ of contactomorphisms such that $A\ast[\varphi]=B$, where $\ast$ stands for the concatenation.  While $\prec$ is always reflexive and transitive, it is antisymmetric if, and only if, $M$ is orderable \cite[Propositions~2.1.A and 2.1.B]{eliashberg2000partially}. Therefore, the orderability from Definition~\ref{def:orderability} is actually about orderability of the universal cover  $\widetilde{\op{Cont}}(M)$. In a similar fashion, one can define a partial order on $\op{Cont}(M)$ if there are no positive (contractible or not) loops of contactomorphisms on $M$. We say that a contact manifold $M$ is \emph{strongly orderable}\footnote{There is no consensus on the meaning of 'strongly orderable' in the literature.  In this article, the phrase has the same meaning as in \cite{casals2016strong}. On the other hand, 'strong orderability' from \cite{liu2020positive} refers to a different quality of contact manifolds.} if it does not admit positive loops of contactomorphisms. Obviously, strong orderability implies orderability.  The significance of orderability in contact geometry goes beyond contact non-squeezing. For instance, the following theorem by Albers, Fuchs, and Merry \cite{albers2015orderability} relates orderability to the famous Weinstein conjecture.

\begin{theorem}[Albers-Fuchs-Merry]
	The Weinstein conjecture holds for any contact manifold that is not strongly orderable. If a contact manifold is not orderable, then every contact form on it has a contractible closed Reeb orbit.
\end{theorem}

Another example is work of Colin and Sandon \cite[Proposition~3.2]{colin2015discriminant} in which they show that the orderability of a contact manifold $M$ is equivalent to the non-degeneracy of the oscillation pseudo-norm on $\widetilde{\op{Cont}}_0(M)$. Along these lines, recent work by Allais and Arlove \cite{allais2023spectral} reinterprets orderability in terms of existence of spectral selectors. The role of positive loops of contactomorphisms in contact geometry is further emphasised by work of Hern\'{a}ndez-Corbato and Mart\'{i}nez-Aguinaga \cite{hernandez2024topology} in which they relate homotopy groups of the loop space of $\op{Cont}(M)$ and the homotopy groups of the space of positive loops of contactomorphisms on $M$. Partially inspired by problems arising in constructions of contact structures in \cite{borman2015existence}, Cieliebak, Eliashberg, and Polterovich introduced and studied the concept of orderability up to conjugation \cite{cieliebak2017contact}. This concept is further explored in the PhD thesis by De Groote \cite{de2019orderability}.

The question of orderability would have been quickly resolved if there existed non-trivial non-negative loops of compactly supported contactomorphisms on the standard contact $\R^{2n+1}$. Namely, using such a loop and a contact Darboux chart one would be able to prove that every contact manifold is at least not strongly orderable. However, there are no such loops for $\R^{2n+1}$ \cite{bhupal2001partial}. This claim is reasserted by (strongly) orderable examples that follow (see Example~\ref{ex:sorbcrit} and Theorem~\ref{thm:weigel}). Similarly, all non-negative loops of compactly supported contactomorphisms of $\R^{2n}\times\mathbb{S}^1$ are trivial, i.e. they are constant loops. In other words, the standard contact $\R^{2n+1}$ and $\R^{2n}\times\mathbb{S}^1$ are strongly orderable, where the orderability here should be understood in the compactly supported category. Now, we list examples of orderable and non-orderable contact manifolds.

\begin{ex}
Lemma~\ref{lem:EposloopS} claims that the standard $\mathbb{S}^{2n+1}$ is not orderable for $n\geqslant 1$. On the other hand, Theorem~\ref{thm:n=1} and Section~\ref{sec:roleloops} imply that $\mathbb{S}^1$ is orderable\footnote{The orderability of $\mathbb{S}^1$ can also be directly verified.}. None of the standard spheres is strongly orderable because the Reeb flow for the standard contact form gives rise to a positive loop of contactomorphisms. This particular loop, however, is never contractible \cite{casals2016chern}. 
\end{ex}

The non-orderability of standard spheres (except, perhaps, of $\mathbb{S}^3$) is a special case of the following result: for every Liouville manifold $W$ (of finite type) and $k\geqslant 2$, the ideal boundary of $W\times\mathbb{C}^k$ is not orderable \cite[Section~3]{eliashberg2006geometry}. The case where $k=1$ is not completely understood.

\begin{qu}
	Does there exist a Liouville manifold $W$ such that $W\times\mathbb{C}$ has orderable ideal boundary?
\end{qu}

As opposed to the higher dimensional spheres, the projective spaces ${\R}P^{2n+1}$ are orderable \cite{givental1990nonlinear,givental1991nonlinear}. So are, by the PhD thesis of Milin \cite{milin2008orderability}, the standard contact lens spaces associated to prime numbers (see also \cite{sandon2011equivariant,granja2021givental}).

\begin{ex}\label{ex:lens}
Let $p$ be a prime number and let $L^{2n-1}_p:= \mathbb{S}^{2n-1}/\mathbb{Z}_p$ be the standard contact lens space obtained by quotienting the standard contact sphere by rotations of the Hopf fibres. Then, $L_p^{2n-1}$ is not orderable.
\end{ex}

By taking $p=2$, Example~\ref{ex:lens} recovers the orderability of the projective spaces. Another source of orderable contact manifolds are unit cotangent bundles of closed smooth manifolds.

\begin{ex}\label{ex:cotangent}
For every closed manifold $M$, its unit cotangent bundle $S^\ast M$ is orderable \cite{eliashberg2006geometry,chernov2010non}.
\end{ex}

In fact, every contact manifold that is fillable by a Liouville domain with non-zero symplectic homology is orderable. This was first shown by Chantraine, Colin, and Rizell in \cite{chantraine2019positive}, strengthening a previous result by Albers and Merry \cite{albers2018orderability}. The result can be equivalently phrased in terms of Rabinowitz Floer homology: if Rabinowitz Floer homology of a contact manifold inside its Liouville filling is non-vanishing, then the contact manifold is orderable. By work of Cieliebak, Frauenfelder, Oancea \cite{cieliebak2010rabinowitz}, and Ritter \cite[Theorem~13.3]{ritter2013topological}, non-vanishing of Rabinowitz Floer homology is equivalent to non-vanishing of symplectic homology.

\begin{ex}\label{ex:SH}
If a Liouville domain $W$ satisfies $SH(W)\not=0$, then the contact manifold $\partial W$ is orderable.
\end{ex} 

Viterbo's isomorphism between symplectic homology $SH(T^\ast M; \mathbb{Z}_2)$ and loop space homology $H(\Lambda M; \mathbb{Z}_2)$ \cite{viterbo2003functors}, see also \cite{abbondandolo2006floer,salamon2006floer,abouzaid2015symplectic}, presents Example~\ref{ex:cotangent} as a special case of Example~\ref{ex:SH}. Example~\ref{ex:SH} (and Example~\ref{ex:cotangent} as well) also implies orderability of ${\R}P^{3}$ because ${\R}P^{3}$ is contactomorphic to the unit cotangent bundle $S^\ast\mathbb{S}^2$ of the sphere. The other projective spaces of dimension greater than 1 are not even Liouville fillable \cite{ghiggini2022symplectic,zhou2021projective,zhou2024fillings} and, therefore, are instances of orderability phenomena beyond the scope of Examples~\ref{ex:SH} and \ref{ex:cotangent}.
Another special case of Example~\ref{ex:SH} is the orderability of the link of a polynomial $p\in\mathbb{C}[z_0,\ldots, z_n]$ with an isolated singularity and with the positive Milnor number \cite{kwon2016brieskorn}. In particular, many Brieskorn manifolds are orderable although never strongly orderable. The following question inquires about the opposite of Example~\ref{ex:SH}.

\begin{qu}
	Does there exist a Liouville domain with vanishing symplectic homology and with orderable boundary?
\end{qu}

By work of Albers and Kang \cite[Theorem~1.1(e)]{albers2023rabinowitz} and of Bae, Kang , and Kim \cite[Corollary~1.10]{bae2024rabinowitz}, it is known that many prequantization bundles are orderable. Since the prequantization bundles have periodic Reeb flows, they are never strongly orderable. The following example, due to \cite[Example~9.2]{chernov2010non} and \cite{albers2012variational}, is a criterion for strong orderability of unit cotangent bundles. 

\begin{ex}\label{ex:sorbcrit}
Let $M$ be a closed smooth manifold. Assume that one of the following conditions is satisfied:
\begin{enumerate}
	\item the fundamental group $\pi_1(M)$ is infinite,
	\item the fundamental group $\pi_1(M)$ is finite and the cohomology ring $H^\ast(M; \mathbb{Q})$ has at least two generators.
\end{enumerate}
Then, the unit cotangent bundle $S^\ast M$ is strongly orderable. In other words, $M$ has no positive loops of contactomorphisms.
\end{ex}

Along these lines, Weigel \cite{weigel2015orderable} showed that the boundary of a Liouville domain is strongly orderable if its filtered Rabinowitz Floer homology has superlinear growth. As an application, he obtained the following beautiful result.

\begin{theorem}[Weigel]\label{thm:weigel}
	Let $(M, \xi)$ be a Liouville fillable contact manifold of dimension at least 7. Then, it is possible to obtain a strongly orderable contact structure on $M$ by modifying $\xi$ in a contact Darboux chart.
\end{theorem}

\section{Orderability and quasimorphisms}

Given a group $G$, a quasimorphism on $G$ is a function $\mu:G\to\R$ such that
\[\abs{\mu(ab)-\mu(a)-\mu(b)}\leqslant C\]
for all $a, b\in G$ and for some $C\in\R^+$, independent of $a$ and $b$. Quasimorphisms can be thought of as homomorphisms up to a bounded error.
Quasimorphisms are a useful tool for studying groups of transformations \cite{entov2004commutator,fukaya2019spectral,gambaudo2004commutators,ghys2001groups,ghys2007knots,polterovich2006floer,shelukhin2014action,usher2011deformed}, particularly when these groups happen to be perfect\footnote{Recall that perfect groups admit no non-zero homomorphisms.}. 
 A quasimorphism is called \emph{homogeneous} if $\mu(a^k)= k\mu(a)$ for all $a\in G$ and all $k\in\mathbb{Z}$. Denote by $\op{Cont}_0(M)$ the identity component of the group of contactomorphisms (of a contact manifold $M$). A quasimorphism $\mu:\widetilde{\op{Cont}}_0(M)\to R$ is said to be \emph{monotone} if $A\prec B$ implies $\mu(A)\leqslant \mu(B)$ for all $A, B\in \widetilde{\op{Cont}}_0(M)$. As in Section~\ref{sec:ord}, $\prec$ denotes the preorder relation furnished by non-negative paths of contactomorphisms. The following theorem from \cite{eliashberg2000partially} establishes a link between orderability and the existence of monotone quasimorphisms on the universal cover of the group of contactomorphisms. In \cite[Section~1.3.E]{eliashberg2000partially}, the theorem was proven in the special case of Givental's quasimorphism on $\R P^{2n+1}$. The same proof, however,  works in general, see \cite[Theorem~1.28]{borman2015quasimorphisms}. 

\begin{theorem}[Eliashberg-Polterovich]\label{thm:quasi}
A contact manifold $M$ is orderable if there exists a monotone non-zero homogeneous quasimorphism on $\widetilde{\op{Cont}}_0(M).$
\end{theorem}

In the view of Theorem~\ref{thm:quasi}, quasimorphisms constructed in \cite{borman2015quasimorphisms} and \cite{zapolsky2020quasi}, see Corollary~1.29 and Theorem~1.3 in \cite{borman2015quasimorphisms} and Corollary~1.3 in \cite{zapolsky2020quasi}, provide further examples of orderable contact manifolds.

\section{Orderability of overtwisted contact manifolds} 

The relation between overtwistedness and orderability is not well understood. Not a single example of an overtwisted contact manifold is known for which the orderability question is resolved. In other words, the following two questions are completely open.

\begin{qu}\label{qu:ordover}
Does there exist an orderable overtwisted contact manifold?
\end{qu}

\begin{qu}\label{qu:nonordover}
Does there exist a non-orderable overtwisted contact manifold?
\end{qu}

The first step towards the answer to these questions was made by Casals, Presas, and Sandon in \cite{casals2016small}, where they showed that if there exists a positive loop of contactomorphisms on an overtwisted 3-dimensional contact manifold then its contact Hamiltonian cannot be arbitrarily small in the $C^0$, and also in the $L^1$, sense. This result has been extended to higher-dimensional overtwisted contact manifolds in \cite{hernandez2020tight} by Hern\'andez-Corbato, Mart\'{i}n-Merch\'{a}n, and Presas. In \cite{casals2016strong}, Casals and Presas constructed positive loops of contactomorphisms on certain overtwisted contact manifolds, including overtwisted $\mathbb{S}^{3}$ and $\mathbb{S}^1\times\mathbb{S}^2$. Therefore, \cite{casals2016strong} answers affirmatively the analogue of Question~\ref{qu:nonordover} for strong orderability. In fact, by work of Liu \cite{liu2016thesis,liu2020positive}, overtwisted contact manifolds are never strongly orderable.

Another partial answer to Questions~\ref{qu:ordover} and \ref{qu:nonordover} is given by Borman, Eliashberg, and Murphy in \cite{borman2015existence}. Instead of considering the preorder $\prec$, they consider a weaker relation $\lessapprox$, whose precise definition will be given later in this section. `Weaker relation' means that $A\prec B$ implies $A \lessapprox B$. In particular, if $\prec$ is not asymmetric then neither is $\lessapprox$. Therefore, non-orderability (with respect to $\prec$) implies non-orderability with respect to $\lessapprox$. The next theorem from \cite{borman2015existence} answers the analogues of Questions~\ref{qu:ordover} and \ref{qu:nonordover} for the preorder relation $\lessapprox$.

\begin{theorem}[Borman-Eliashberg-Murphy]
Every closed overtwisted contact manifold is not $\lessapprox$-orderable.
\end{theorem}

It is actually not known whether non-orderability of $\lessapprox$ implies non-orderability \cite[page 358]{borman2015existence}. Now, we define the preorder relation $\lessapprox$ on $\widetilde{\op{Cont}}(M)$. Let $\varphi^h$ be the contact isotopy furnished by a 1-periodic contact Hamiltonian $h_t: M\to\R$. Denote by $V^+(\varphi^h)$ the subset of $M\times T^\ast \mathbb{S}^1$ given by
\[ V^+(\varphi^h):=\left\{ (x,v,t)\::\: v+ h_t(x)\geqslant 0 \right\}\subset M\times T^\ast\mathbb{S}^1. \]
Here, we tacitly identified $T^\ast\mathbb{S}^1$ with $\R\times\mathbb{S}^1$ and denoted by $v$ and $t$ the coordinates of $\R\times\mathbb{S}^1$. Elements $A, B\in \widetilde{\op{Cont}}(M)$ are said to satisfy the relation $A\lessapprox B$ if there exist representatives $\varphi^h$ and $\varphi^f$ respectively of $A$ and $B$ and a contact isotopy 
$$\Phi_t: M\times T^\ast\mathbb{S}^1\to M\times T^\ast\mathbb{S}^1$$
such that $\Phi_0=\op{id}$ and such that $\Phi_1(V^+(\varphi^h))\subset V^+(\varphi^f)$. In \cite{eliashberg2000partially}, it was proved that indeed $A \prec B$ implies $A \lessapprox B$.

\section{Positive isotopies of Legendrians}

Closely related to orderability is the question of existence of positive loops of Legendrians.

\begin{defn}\label{def:leg}
	Let $M$ be a cooriented contact manifold. An isotopy $\{L_t\}_{t\in[0,1]}$ of Legendrian submanifolds of $M$ is called positive if there exists a parametrization
	\[\iota : L\times[0,1]\to M\]
	such that $\iota(L\times\{t\})= L_t$ and such that the vector $\partial_t\iota(x,t)$ points in the positive direction, that is $\alpha(\partial_t\iota(x,t))>0$ for some contact form $\alpha$ on $M$ and for all $x$ and $t$.
\end{defn}

If $L_0=L_1$ in Definition~\ref{def:leg}, then we say that $\{L_t\}$ is a positive loop of Legendrians. The notions of non-negative Legendrian isotopies and of non-negative loops of Legendrians are defined analogously (by replacing the condition $\alpha(\partial_t\iota(x,t))>0$ with $\alpha(\partial_t\iota(x,t))\geqslant 0$). If $\Lambda$ denotes the space of Legendrians in $M$ that are isotopic to a given Legendrian $L\subset M$, then non-negative Legendrian isotopies give rise to a preorder relation on $\Lambda$. Namely, $L_0\prec L_1$  for $L_0, L_1\in\Lambda$ if, and only if, there exists a non-negative Legendrian isotopy $\{L_t\},  t\in[0,1]$ from $L_0$ to $L_1$. The relation $\prec$ on $\Lambda$ is a partial order, if, and only if, there are no positive loops of Legendrians in $\Lambda$ \cite[Proposition~4.5]{chernov2016universal}. If there are no contractible positive loops of Legendrians in $\Lambda$, then $\prec$ induces a partial order on the universal cover $\tilde{\Lambda}$ of $\Lambda$ \cite[Proposition~4.5]{chernov2016universal}. Obviously, for a non-orderable contact manifold $M$, the induced relation on $\tilde{\Lambda}$ is never a partial order because a contractible positive loop $\varphi_t: M\to M$ of contactomorphisms furnishes a contractible positive loop $\{\varphi_t(L)\}$ of Legendrians in $\Lambda$. Similarly, $\prec$ is not a partial order on $\Lambda$ for any $L$,  if $M$ is not strongly orderable. 

Now, we list results regarding existence and non-existence of positive Legendrian paths. By work of Colin, Ferrand, and Pushkar, the 1-jet spaces of closed manifolds do not admit positive loops of Legendrians containing the 1-jet extensions of functions \cite{colin2017positive}.

\begin{theorem}[Colin-Ferrand-Pushkar]
	Let $N$ be a smooth closed manifold. Then, there does not exist a positive Legendrian loop in $J^1(N)$ based at the 1-jet extension of a function.
\end{theorem}

On the other hand, there are positive loops of Legendrians in $J^1(\mathbb{S}^1)$ that are not based at the 1-jet extensions of functions \cite[Theorem~3]{colin2017positive}. The following theorem by Chernov and Nemirovski \cite{chernov2010non} was used to show that $S^\ast M$ is orderable for every closed manifold $M$. This was an enhancement of \cite[Theorem~1.18]{eliashberg2006geometry} where it was shown that $S^\ast M$ is orderable if $\pi_1(M)$ is either finite or has infinitely many conjugacy classes.

\begin{theorem}[Eliashberg-Kim-Polterovich, Chernov-Nemirovski]\label{thm:ChN}
	Let $M$ be a smooth connected manifold of dimension at least 2. Assume that the universal cover of $M$ is an open manifold. Then, there are no non-negative Legendrian isotopies connecting two different fibres of $S^\ast M$. 
\end{theorem}
In \cite{guillermou2012sheaf} (see Corollary~4.10), Guillermou, Kashiwara, and Schapira reprove the theorem using microlocal analysis. As a consequence of Theorem~\ref{thm:ChN}, there are no positive Legendrian loops on $S^\ast M$ starting at a fibre \cite[Corollary~8.1]{chernov2010non} if the universal cover of $M$ is an open manifold. Theorem~\ref{thm:ChN} generalizes previous results regarding manifolds covered by $\R^n$, see \cite[Theorem~2]{colin2017positive} and \cite[Corollary~6.2]{chernov2010legendrian}. An analogue of the theorem holds for conormal bundles of simply connected closed submanifolds of codimension at least 2 \cite[Theorem~4.9]{chernov2016universal}. There are manifolds $M$ that violate the conclusion of Theorem~\ref{thm:ChN}. For instance, if $M$ is a compact rank-one symmetric space, that is $\mathbb{S}^n, \R P^n, \mathbb{C}P^n, \mathbb{H}P^n,$ or the Cayley plane $\mathbb{O}P^2$,  then the cogeodesic flow on $S^\ast M$ eventually takes a fibre to another fibre. This is most easily seen on the example of $M=\mathbb{S}^n$: the cogeodesic flow at time $\pi$ takes the fibre above the north pole to the fibre above the south pole. On the other hand, there are never contractible positive Legendrian loops starting at a fibre of a sphere cotangent bundle \cite[Theorem~1.1]{chernov2016universal}.

\begin{theorem}[Chernov-Nemirovski]
	Let $M$ be a smooth manifold. Then, there does not exist a positive contractible Legendrian loop on $S^\ast M$ that starts at a fibre.
\end{theorem}

A similar statement holds for conormal bundles of a connected closed submanifold of codimension at least 2 \cite[Theorem~4.10]{chernov2016universal}. The following result by Frauenfelder, Labrousse, and Schlenk  \cite[Theorem~1.13]{frauenfelder2015slow} shows that the class of closed manifolds $M$ for which $S^\ast M$ admits a positive Legendrian loop starting at a fibre is very restricted.
	
\begin{theorem}[Frauenfelder-Labrousse-Schlenk]\label{thm:FLS}
	Let $M$ be a closed connected manifold of dimension at least 2. If $S^\ast M$ admits a positive Legendrian loop starting at a fibre, then $\pi_1(M)$ is finite and the cohomology ring $H^\ast(M;\mathbb{Z})$ is isomorphic to that of a compact rank-one symmetric space.
\end{theorem}

Along these lines, Dahinden \cite{dahinden2018bott} showed that $M$ has to be either simply connected or homotopy equivalent to $\R P^n$ if, in the situation of Theorem~\ref{thm:FLS}, the Legendrian at which the positive loop is based is disjoint form all the other Legendrians in the loop. In \cite{chantraine2019positive}, Chantraine, Colin, and Rizell provide plenty of non-existence results for positive Legendrian loops using Floer-theoretic methods. In particular, they prove the following `open-string' analogue of Example~\ref{ex:SH}.

\begin{theorem}[Chantraine-Colin-Rizell]
	Let $W$ be a Liouville domain and  let $L\subset \partial W $ be a Legendrian that admits an exact Lagrangian filling $N\subset W$. If the wrapped Floer homology of $N$ is non-vanishing, then the Legendrian $L$ is not contained in a contractible positive loop of Legendrians.
\end{theorem}

The following result by Liu \cite{liu2016thesis,liu2020positive} is well in line with flexible behaviour of loose Legendrians. See \cite{pancholi2018simple} for further constructions of positive Legendrian loops, that in particular partially recover the result of Liu.

\begin{theorem}[Liu]
Let $M$ be a contact manifold of dimension at least 5 and let $L\subset M$ be a loose Legendrian. Then, there exists a contractible positive loop of Legendrians containing $L$.
\end{theorem}

\section*{Acknowledgements}
	I am grateful to Roger Casals and Leonid Polterovich for valuable feedback.
\printbibliography
\end{document}